\documentclass[12pt,a4paper,english,final]{article}

\usepackage{showkeys}
\usepackage{amssymb}
\usepackage{amsmath}
\usepackage{theorem}
\usepackage[applemac]{inputenc}
\usepackage{enumerate,xspace}
\usepackage{calc}
\usepackage[usenames]{color}
\usepackage{colortbl}

\newtheorem{theorem}{Theorem}[section]
\newtheorem{lemma}[theorem]{Lemma}

\newtheorem{proposition}[theorem]{Proposition}
\newtheorem{corollary}[theorem]{Corollary}
\newtheorem{remark}[theorem]{Remark}
\newtheorem{example}[theorem]{Example}
\theoremstyle{break}

\newtheorem{corollarybr}[theorem]{Corollary}

\newenvironment{proof}{\textbf{Proof.}~}{\hfill$\square$}

\newenvironment{vd}{\noindent\color{blue} VD }{}

\newenvironment{jl}{\noindent\color{red} JL }{}

\newcommand{\ov}[1]{\overline{#1}}

\newcommand{\abs}[1]{\left|\mathinner{#1}\right|}
\newcommand{\Abs}[1]{\left\Vert\mathinner{#1}\right\Vert}

\newcommand{\gen}[1]{\left< \mathinner{#1} \right>}

\newcommand{\set}[2]{\left\{\, \mathinner{#1}\vphantom{#2}\; \left|\; \vphantom{#1}\mathinner{#2} \right.\,\right\}}
\newcommand{\oneset}[1]{\left\{\, \mathinner{#1} \,\right\}}

\newcommand{\N}{\mathbb{N}}
\newcommand{\Z}{\mathbb{Z}}

\newcommand{\Oh}{\mathcal{O}}

\newcommand{\IFF}{if and only if\xspace}
\newcommand{\homo}{homomorphism\xspace}
\renewcommand{\phi}{\varphi}
\newcommand{\eps}{\varepsilon}

\newcommand{\alp}{\alpha}
\newcommand{\bet}{\beta}
\newcommand{\gam}{\gamma}
\newcommand{\del}{\delta}

\newcommand\RAS[2]{\overset{#1}{\underset{#2}{\Longrightarrow}}}
\newcommand\ra{\longrightarrow}

\newcommand\LAS[2]{\overset{#1}{\underset{#2}{\Longleftarrow}}}

\newcommand\DAS[2]{\overset{#1}{\underset{#2}{\Longleftrightarrow}}}

\newcommand{\Breduced}{Britton-reduced\xspace}
\newcommand{\Breduction}{Britton reduction\xspace}
\newcommand{\geo}{geo\-desic\xspace}
\newcommand{\BSpq}{\mathrm{\bf{BS}}(p,q)}
\newcommand{\llex}{length-lexi\-co\-graphi\-cal\xspace}
\newcommand{\llexnf}{length-lexi\-co\-graphi\-cal nor\-mal form\xspace}

\newcommand{\sv}{standard valley\xspace}
\newcommand{\confree}{context-free\xspace}

\renewcommand{\ll}{\mathrel{\leq_{\mathrm{ll}}}}
\newcommand{\lldel}{\mathrel{\leq_{\Delta}}}
\newcommand{\llnf}[1]{\mathop{\mathrm{llnf}}(#1)}

\newcommand{\di}{difficult\xspace}
\newcommand\Bpnf{Britton peak normal form\xspace}
\newcommand\peaknf{peak normal form\xspace}
\newcommand{\pnf}[1]{\mathop{\mathrm{pnf}}(#1)}
\newcommand{\sh}{semi-horocyclic\xspace}

\renewcommand{\t}{\mathrm{time}}
\newcommand{\T}{\mathrm{Time}}

\newcommand{\FP}{\Z * \oneset{t,T}^*}

\newcommand{\gq}{\sim}
\newcommand{\sq}{=}

\newcounter{BStmpa}
\newcounter{BStmpb}
\newcounter{BStmpc}
\newcounter{BSdx}
\newcounter{BSdy}
\newcounter{BSheight}
\newcounter{BSx}
\newcounter{BSy}

\definecolor{BSgray}{rgb}{0.9,0.9,0.9}
\definecolor{BSblack}{rgb}{0.0,0.0,0.0}

\newenvironment{BS}[2]
{
	\setlength{\unitlength}{1mm}
	\setcounter{BSdx}{5*#1+10}
	\setcounter{BSdy}{10*#2+14}
	\setcounter{BSheight}{#2}
	\begin{picture}(\value{BSdx},\value{BSdy})
	\color{BSgray}
	\setcounter{BStmpa}{7}
	\loop
		\ifnum\value{BStmpa}<\value{BSdy}
		\put(0,\value{BStmpa}){\line(1,0){\value{BSdx}}}
		\addtocounter{BStmpa}{10}
	\repeat
	\put(0,0){\line(1,0){\value{BSdx}}}
	\color{BSblack}
	\setcounter{BSx}{5}
	\setcounter{BSy}{7}
}
{
	\thinlines
	\put(0,0){\framebox(\value{BSdx},\value{BSdy}){}}
	\end{picture}
}

\newcommand{\BSt}
{
	\put(\value{BSx},\value{BSy}){\line(1,2){5}}
	\addtocounter{BSx}{5}
	\addtocounter{BSy}{10}
}
\newcommand{\BST}
{
	\put(\value{BSx},\value{BSy}){\line(1,-2){5}}
	\addtocounter{BSx}{5}
	\addtocounter{BSy}{-10}
}
\newcommand{\BSa}[1]
{
	\setcounter{BStmpa}{\value{BSx}-1}
	\setcounter{BStmpb}{\value{BSy}+2}
	\put(\value{BSx},\value{BSy}){\circle*{1.0}}
	\put(\value{BStmpa},\value{BStmpb}){\ensuremath{#1}}
}
\newcommand{\BSoff}[2]
{
	\addtocounter{BSx}{5*#1}
	\addtocounter{BSy}{10*#2}
}
\newcommand{\BSdesc}[1]
{
	\put(5,1.5){#1}
}

\newcommand{\BSheights}[1]
{
	\setcounter{BStmpa}{-#1}
	\setcounter{BStmpb}{6}
	\setcounter{BStmpc}{\value{BSheight}-#1+1}
	\loop
		\ifnum\value{BStmpa}<\value{BSheight}
		\put(1,\value{BStmpb}){\theBStmpa}
		\addtocounter{BStmpa}{1}
		\addtocounter{BStmpb}{10}
	\repeat
	\addtocounter{BSx}{3}
}

\begin{document}


\title{On Computing Geodesics\\in Baumslag-Solitar Groups}
\author{Volker Diekert \qquad J\"urn Laun \\[5mm]
  Universit{\"a}t Stuttgart, FMI \\
  Universit{\"a}tsstra{\ss}e 38 \\
  D-70569 Stuttgart, Germany \\[5mm]
  \texttt{$\{$diekert$,$laun$\}$@fmi.uni-stuttgart.de}}

\maketitle 

\abstract
{
We introduce the \emph{\peaknf}for 
elements of the 
Baumslag-Solitar groups $\BSpq$. This normal form is 
very close to the \llex normal form, but more symmetric. 
Both normal forms
are geodesic. This means the normal form of an element $u^{-1}v$ yields 
the shortest path between $u$ and $v$ in the Cayley graph.
For horocyclic elements the \peaknf and the \llexnf coincide. 
The main result of this paper is that we  can compute the \peaknf in polynomial time
if $p$ divides $q$. As consequence we can compute geodesic lengths
in this case. In particular, this gives a partial answer to Question~1 in \cite{elder09some}.

For arbitrary $p$ and $q$ it is possible to compute the \peaknf 
(\llexnf resp.)  also the for elements 
in the horocyclic subgroup and, more generally, for elements which we 
call \emph{hills}. This approach leads to a  linear time reduction 
of the problem of
computing geodesics to the problem of
computing geodesics for \Breduced words where the $t$-sequence 
starts with $t^{-1}$ and ends with $t$. 
}

\section{Introduction}\label{sec:intro}
Baumslag-Solitar groups were introduced in \cite{baumslag62some}
and they enjoy many remarkable properties, see e.g. Lyndon and Schupp \cite{ls77}.
The Baumslag-Solitar group $\BSpq$ is a  one-relator group defined by
\begin{equation*}
	\BSpq:=\langle a,t\ \vert\ ta^pt^{-1}=a^q\rangle.
\end{equation*}

The word problem is decidable in linear time, but it is still not known
how to compute the geodesic length of elements efficiently.
Polynomial time algorithms for 
this problem were known only for horcyclic elements, \cite{fredenadams09,fcs09} or in the case where 
$\BSpq$ is solvable, i.e., the case where $p=1$, see  \cite{elder09linear}. As usual, a horocyclic element is an element 
of the subgroup $\gen{a}$ generated by $a$. 
More precisely, the paper of Murray Elder \cite{elder09linear} 
presents a linear time 
algorithm how to compute \geo{}s for all words, when
$\BSpq$ is solvable. 

The result of Elder is the starting
point for our paper and we generalize  his result to the 
case where $p$ divides $q$, thereby giving  a partial answer to Question~1 in \cite{elder09some}.

For this purpose we introduce the notion of \emph{\peaknf.}
The \peaknf is geodesic. Thus, it represents an element by a
geodesic word in the Cayley graph of $\BSpq$.
There can be exponentially many different \geo{}s, 
the \peaknf chooses a unique one; and  it is defined by a natural 
condition. For horocyclic elements the \peaknf and 
\llexnf coincide. 

Our main result is Theorem~\ref{thm:main}; it states that 
{\peaknf}s can be computed in quadratic time, if $p$ divides $q$. Actually, we have a more 
precise result, which yields Elder's linear time
bound for $p=1$. 

Our technique relies on the 
fact  that 
\llex normal forms of horocyclic elements can be computed 
in linear time. See also \cite{elder09linear,fredenadams09,fcs09}
for similar approaches. 
We also  extend the linear time result to elements which we call 
\emph{hills}. These are words whose \Breduction 
can be written in the form
$$\bet_1 t \cdots \bet_k t \alp_0 t^{-1} \alp_1 \cdots \alp_\ell t^{-1}$$
where all $\alp_i$ and $\bet_j$ are horocyclic. Moreover, we give 
a linear time reduction of the problem of 
computing geodesics to that of
computing geodesics for \Breduced words where the $t$-sequence 
starts with $t^{-1}$ and ends with $t$. 

\section{Notation and preliminaries}
Throughout the paper, let $1\le p<q$ be fixed positive integers. 
By $\alp, \alp_i, \bet, \bet_i, \gamma, \gamma_i, \delta, \mu, \nu, 
\rho, \sigma, \tau$
we always mean integers. We reserve special fixed letters
$t$, $T$, $a$, and $A$. By  $\theta, \theta_i$ we mean
either $t$ or $T$. 

The alphabet $\oneset{t,T,a,A}$ is ordered 
by putting  $t<T< a < A$, and it is equipped with an involution
by $\ov a = A$, $\ov A = a$, $\ov t = T$, and $\ov T = t$. 

The involution is extended to words by 
$$ \ov{a_1 \cdots a_m} =  \ov{a_m} \cdots \ov{a_1} \text{ for } 
a_i \in \oneset{t,T,a,A}.$$
We read  $A= a^{-1}$ and $T= t^{-1}$ in the Baumslag-Solitar group:
\begin{equation*}
	\BSpq:=\langle a,t\ \vert\ ta^pt^{-1}=a^q\rangle.
\end{equation*}
Every word $w \in \oneset{t,T,a,A}^*$ can be read as $w \in \BSpq$ 
and, of course,  $\ov w = w^{-1}\in \BSpq$. 

The Baumslag-Solitar group $\BSpq$ is an HNN-extension of $\Z=\langle a\rangle$ with stable letter $t$, and the  HNN-extension
is defined by the canonical \homo between the subgroups 
$p\Z$ and $q\Z$ mapping the subgroup generator $p$ to $q$.
The  group $\BSpq$
can also be defined by a confluent and terminating string rewriting system 
$BS$
over the alphabet $\oneset{t,T,a,A}$
with the following rules, where $0$ denotes the empty word:
\begin{align*}
aA \ra 0& \quad \quad Aa \ra 0\\
tT \ra 0& \quad \quad Tt \ra 0\\
a^qt\ra t a^p & \quad \quad At\ra a^{q-1}t A^p  \\
 a^pT\ra  Ta^q & \quad \quad AT\ra a^{p-1}T A^q 
\end{align*}
Termination is not completely obvious, 
but the proof is standard with string rewriting techniques 
as explained e.g.~in the textbook \cite{jan88eatcs}. 

The 
rewriting system $BS$ defines a congruence relation 
$$\DAS*{BS} \subseteq \oneset{t,T,a,A}^* \times \oneset{t,T,a,A}^*.$$ 
This gives  normal forms by computing $w \RAS*{BS} \hat w$
such that $\hat w$ is irreducible. Since $BS$ is confluent, two words $u$ and $v$ are equal in 
$\BSpq$ \IFF $u\DAS*{BS}v$ \IFF there is some $\hat w$ with 
$u\RAS*{BS}\hat w\LAS*{BS}v$. Moreover, if $\hat w$ is irreducible, it is 
uniquely defined by $u$. 

In the following we write $u\gq v$ if $ u\DAS*S v$. Thus, 
$u=v$ means identity as strings, whereas $u\gq v$ means 
equality in $\BSpq$. 

The problem is that the length of $\hat w$ can be exponential
in the length of $w$. To see this, consider e.g. $p=1$ and the
word $t^n a T^n$. Its irreducible descendant has length $q^n$.

The purpose of the system $BS$ is therefore mainly to provide simple 
proofs  for all basic properties about Baumslag-Solitar groups, see
also \cite{ddm09} for more background about this approach. 

We investigate the problem of computing geodesics, i.e., given a word $w$ over the generators $a$ and $t$ and their inverses
$A= a^{-1}$ and $T= t^{-1}$, find a shortest word $g(w)$ over the  alphabet $\oneset{t,T,a,A}$
such that $w$ and $g(w)$ describe the same element of the group $\BSpq$. 
The \geo word $g(w)$ is not unique, but its length $\abs{g(w)} \in \N$ 
is well-defined. It is called the \emph{\geo length} of $w$. 

There might be exponentially many different geodesic words $g(w)$ for a word $w$. So we are interested 
in unique normal forms as well. 
The \emph{\llex} linear order is defined
on words $u$ and $v$ by letting $u\ll v$ if either $\abs{u}<\abs{v}$ 
or first $\abs{u}=\abs{v}$ and second $u$ is not 
behind $v$ in the lexicographical order, 
which is defined by $t<T<a<A$. 

The \emph{\llex}normal form of a word $w$
is  denoted by $\llnf{w}$. It is the first word $v$
in this order satisfying $v\gq w$. Obviously, $\llnf{w}$ is \geo. 
Later we will introduce another (and more symmetric)
\geo normal form: the \peaknf.

Confluence and termination of the rewriting system $BS$ imply that the word problem for $\BSpq$ is solvable. Therefore computing \geo{}s (or any decidable \geo normal form) is possible in finite (though using this na\"ive approach at least exponential) time. On input $w$ just enumerate all words up to length $\abs{w}$ and check if they are equivalent to $w$.

Throughout, we identify the word $a^\alpha$ with the integer $\alp$, and we identify
$A^\alpha$ with $-\alp$. As a consequence,  the free product 
$\FP$ of $\Z$ with the free monoid 
$\oneset{t,T}^*$ becomes both a quotient monoid 
(by $a\mapsto +1$ and $A\mapsto -1$) and at the same time 
a subset of $\oneset{t,T,a,A}^*$, where
$0$ denotes the empty word. For elements of this subset $\FP$ we use normal 
forms. In other words $u \in \FP$ is  always 
represented as  a sequence which alternates between 
integers and elements from $\oneset{t,T}^+$.

More precisely, we identify $\FP$ with the set of sequences
of words which alternate between words in $\oneset{a}^* \cup \oneset{A}^*$
and letters in $\oneset{t,T}$.
Thus, $(42,t,t,3,T,t)$ is allowed, but  $(43,-1,t,t,3,T,t)$
is not, although they denote the same element in $\FP$. 
(This means $\FP$ becomes the regular subset of $\oneset{t,T,a,A}^*$, 
which is defined by forbidding factors $aA$ and $Aa$.)

Note that mapping a word $w\in \oneset{t,T,a,A}^*$ 
to the corresponding word $u \in \FP$ does not increase
the length. 
It is also clear that all geodesic words in $\oneset{t,T,a,A}^*$ belong to the 
subset $\FP\subseteq \oneset{t,T,a,A}^*$.

Every word $u\in \FP$ can uniquely be written  as a sequence 
$$u = \alp_0\theta_1\alp_1\cdots \theta_k\alp_k$$
with $k \geq 0$, $\alp_i \in \Z$, and $\theta_i \in \oneset{t,T}$.
Its length is $$\abs{u} = k + \sum_{i=0}^k \abs{\alp_i}.$$

Let $\Abs{\alp}$ denote the geodesic length of an integer $\alp$. 
We define the \emph{norm} $\Abs{u}$ of $u$ by
$$\Abs{u} =   k + \sum_{i=0}^k \Abs{\alp_i}.$$
Let $\abs{g(u)}$ be the \geo length of $u$. Then we get
$$\abs{g(u)} \leq \Abs{u} = k + \sum_{i=0}^k \Abs{\alp_i}
 \leq \abs {u} = k + \sum_{i=0}^k \abs{\alp_i},
$$
having equalities when $u$ is \geo. 

The objective is therefore to compute on input $u$ a word $v$ which 
minimizes $\Abs{v}$ among all words $v$ with  $u \gq v$. 
In the next section we show that, given a word, we can compute
an equivalent Britton-reduced word in polynomial time. Thus, in order
to compute \geo{}s we only have to consider Britton-reduced words.

Observe  that $\Abs{\alp} = \abs{\alp}$ for an integer $\alp$ implies $\abs{\alp} < 3q$,
and as soon as  $\abs{\alp} \geq 2q$, we find a geodesic using 
letters $t$ (and $T)$. This is trivial, because $\mu q \gq t\mu pT$. 
It is also well-known (and shown in Section~\ref{sec:ghw})
that, if there exists some geodesic using the 
letter $t$, then there is some geodesic $g_1(\alp)$ which starts with 
the letter $t$, and there is some  geodesic $g_2(\alp)$ which ends with 
the letter $T$. 

As a consequence, 
let $w$ be a word and let  $u = \alp_0\theta_1\alp_1\cdots \theta_k\alp_k$ minimize
$\Abs{u}$ for all $u \gq w$.
Consider 
some $0\leq i <k$ where $\theta_{i+1} = t$.
No geodesic of $\alp_i$ can end in a $T$, so 
$\abs{g(\alp_i)} = \abs{\alp_i} < 2q$. 
As we have $(\pm q)t \gq t(\pm p)$ we see that actually
$\abs { \alp_i} < q$. The same happens if $1 \leq j \leq k$  
and
$\theta_{j}= T$, then  $\abs { \alp_j} < q$, too. 

Thus, large values $\abs { \alp_i}$ can be found at local peaks 
of $u$, only. A \emph{local peak} of $u$ is a position
$i \in \oneset{0, \ldots, k}$ such that 
$0\leq i \leq k$ and  $\theta_{i} \neq T$  and $\theta_{i+1} \neq t$.
Note that $0$ is a local peak for $k =0$. For $k >0$ it is 
a local peak if $\theta_1= T$. 

The notion originates from the shape of the path of $u$, 
if reading a $T$ means going one step downwards and
reading a $t$ means going one step upwards, as illustrated
by the following figure for the word $u=ATa^2taTattA^{10}TaTa^2tta^{42}T$:

\begin{center}
\begin{BS}{10}{2}
	\BSdesc{\text{A word with four local peaks}}
	\BSoff{0}{1}
	\BSa{-1}\BST\BSa{2}\BSt\BSa{1}\BST\BSa{1}\BSt\BSt\BSa{-10}\BST\BSa{1}\BST\BSa{2}\BSt\BSt\BSa{42}\BST\BSa{}
\end{BS}
\end{center}

As integers $\alp_i$ with large absolute value $\abs { \alp_i}$ can only be found at local peaks 
of $u$, the number of local peaks is an important
parameter. However, for technical reasons which will not become apparent until later, 
we prefer to count the number of sinks. A sink is dual to a local peak. Formally, 
a \emph{sink} of a word $u = \alp_0\theta_1\alp_1\cdots \theta_k\alp_k$
is a position
$i \in \oneset{0, \ldots, k}$ such that 
$\theta_{i} \neq t$  and $\theta_{i+1} \neq T$.
Note that $0$ is a sink for $k =0$. For $k >0$ it is 
a sink if $\theta_1= t$. The difference between the number of local peaks and the
number of  sinks is bounded by $\pm 1$. 
For example, if $p=1$ (which means that $\BSpq$ is solvable)
\Breduced words may  have 1 or 2 local peaks, but always exactly 1 sink. 
Here are the pictures for words having one sink and one or two local peaks: 

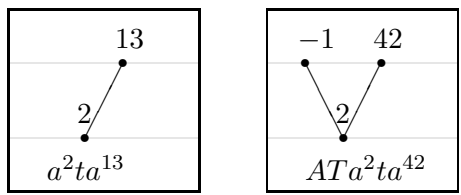
\begin{center}
\begin{BS}{3}{1}
	\BSdesc{\text{$a^2ta^{13}$}}
	\BSoff{1}{0}
	\BSa{2}\BSt\BSa{13}
\end{BS}
\begin{BS}{3}{1}
	\BSdesc{\text{$ATa^2ta^{42}$}}
	\BSoff{0}{1}
	\BSa{-1}\BST\BSa{2}\BSt\BSa{42}
\end{BS}
\end{center}
 
We emphasize this is in a special notation: 
Let $u = \alp_0\theta_1\alp_1\cdots \theta_k\alp_k$ 
be a word. The number of sinks is denoted by $s(u)$ which is defined as 
\begin{equation}\label{eq:lp}
s(u) = \abs{\set{i \in \oneset{0, \ldots, k}}{ i \text{ is a sink}}}.
\end{equation}

\section{\Breduction{}s}
Let $u = \alp_0\theta_1\alp_1\cdots \theta_k\alp_k$ be  a word  in $\FP$. 
A \emph{\Breduction step} means
replacing in $u = \alp_0\theta_1\alp_1\cdots \theta_k\alp_k$ either 
a factor $t \mu p T$ by $ \mu q$ or 
a factor $T \mu q t$ by $ \mu p$, where $\mu \in \Z$. 
After that, we rewrite the new sequence by the corresponding 
word without factors $aA$ and $Aa$ in the free
product $\FP$. So a \Breduction step decreases the number of letters from $\oneset{t,T}$ 
but it may increase the length. If e.g. $\theta_i \alp_i \theta_{i+1} = t \mu p T$, then 
one \Breduction step yields 
$$u_1 = \alp_0\theta_1\alp_1\cdots \theta_{i-1} (\alp_{i-1} 
+ \mu q + \alp_{i+1}) \theta_{i+2} \alp_{i+2}\cdots \theta_k\alp_k.$$
A \Breduction may increase the length by a factor $q/p$, 
but the important point is 
that $\Abs{u_1} \leq \Abs{u}$. Indeed the geodesic length
of $\alp_{i-1} + \mu q + \alp_{i+1}$ is at most 
$\abs{\alp_{i-1}t\mu pT\alp_{i+1}}.$
A \Breduction neither increases the number of local peaks nor the number of sinks.  

A word is \emph{\Breduced,} if no 
\Breduction step is possible. 
As the coefficients 
$\alp_i$ can increase exponentially, from now on we keep
all integers in binary notation. The number of bits remains 
linear in the input length because the sum over $\log_2(\abs{\alp_i})$ is never greater than $q\abs w$. For complexity considerations it is convenient 
to assume that  an arithmetic operation on $n$-bit integers (with $n \in \Oh(\abs w)$) 
can be performed in constant time.  
Then each step in a \Breduction 
needs constant time. So we can produce a \Breduced equivalent
word (with 
integers written in binary) in 
linear time. The number 
of bit operations is actually not worse than quadratic 
(if polylog terms are ignored). 

Britton-reduced forms are far from being unique. On the contrary, there can be exponentially  many 
\Breduced equivalent words. But if 
$w\gq u = \alp_0\theta_1\alp_1\cdots \theta_k\alp_k$ and $u$ is \Breduced,
 then the \emph{$t$-sequence} 
$$(\theta_1, \ldots \theta_k) \in \oneset{t,T}^*$$
is uniquely determined by the word $w$. 
(This can easily be seen with the help of the rewriting 
system $BS$ defined above since applying its rules leaves a 
\Breduced word \Breduced.)

Thus, as a \Breduction never increases the 
norm, we can assume that the input word 
$u = \alp_0\theta_1\alp_1\cdots \theta_k\alp_k$ is Britton 
reduced and the integers $\alp_i$ are written in binary. 
The objective is therefore reduced to minimizing the norm $\Abs{u}$ 
(and to computing geodesic normal forms) for \Breduced words and to computing \geo{}s of horocyclic elements. 

As usual, a word $u\in \oneset{t,T,a,A}^*$ is called  
\emph{horocyclic}, if its image in $\BSpq$ belongs to 
the horocyclic subgroup $\gen{a}$ generated by $a$. We adopt this notion
here. Thus,  a \Breduced word
$u = \alp_0\theta_1\alp_1\cdots \theta_k\alp_k$ is horocyclic
\IFF $k = 0$. As a consequence, membership of 
the horocyclic subgroup can be tested in linear time. 

The next step is to compute geodesics of horocyclic words.

\section{Geodesics for horocyclic words}\label{sec:ghw}
Techniques for computing \geo{}s  for horocyclic words
can also be found elsewhere, see \cite{elder09linear,fredenadams09,fcs09}. Corollary~\ref{allesreg} and
Remark~\ref{rem:ft} have been inspired  
by recent results by Freden et al., see \cite[Prop.~9.1]{fcs09}.
Since our presentation is however quite different from other approaches 
we give full proofs.

The strategy to compute \llex normal forms for 
horocyclic words $w \gq \alp \in \Z$
relies on the following simple observations:
\begin{enumerate}[1.)]
\item  Horocyclic elements commute in $\BSpq$.
\item If $\abs{\alp}$ is small, then we can
use table lookup. 
\item If $\abs{\alp}$ is large, then $\llnf{w}$ begins with $t$, see Lemma~\ref{danksprung}.
\item If  $w \gq t u Tv$, where the indicated letters $t$ and $T$ match in 
a \Breduction, then $v \gq \bet\in \Z$ and $\alp \equiv \bet \bmod q$.
\item  If $w = uw'v$ and  $\llnf{w} = uw''v$, then $\llnf{w'} = w''$.
\end{enumerate}

We begin with a lemma which might be of independent interest.
\begin{lemma}\label{danksprung}
Let $w \gq \alp$ be horocyclic and  $g(w)$ be a 
\geo representation. Then we can write:
$$ g(w) = \bet_{k}t \cdots \bet_1 t\alp_0T\alp_1\cdots T\alp_k.$$
Moreover, there are also \geo representations $g_1(w)$ 
and $g_2(w)$ such that
\begin{align*}
g_1(w) = t \cdots  t\alp_0T(\alp_1 + \bet_1) \cdots T(\alp_k + \bet_{k})\\
g_2(w) = (\alp_k + \bet_{k})t \cdots  (\alp_1 + \bet_1)t\alp_0T \cdots T
\end{align*}
We also find such geodesic representations with $k\geq 1$  as soon as 
$\abs \alp \geq 2q$. 
\end{lemma}

\begin{proof} 
The shape $ \alp \gq g(w) = \bet_{k}\theta'_k \cdots \bet_1 \theta'_1\alp_0\theta_1\alp_1\cdots \theta_k\alp_k$ is obvious because
in the \Breduction all $t$ and $T$ must vanish. 
Clearly, $a^{2q} \gq ta^{2p}T$ and $A^{2q} \gq tA^{2p}T$, so 
the assertions are  trivial for $k=0$. Let 
$k \geq 1$.   

\begin{enumerate}[1.)]
\item Let $g(w) = uv$ be the product of shorter horocyclic words.
If both $u$ and $v$ have \geo representations  using the letter $t$
(or $T$), then we can write $g(w) \gq g_2(u)g_1(v)$ with $\abs{g(w)} = \abs{g_2(u)g_1(v)}$.
But $g_2(u)g_1(v)$
is not geodesic, because the first letter of 
$g_1(v)$ is $t$ and the last letter of 
$g_2(u)$ is $T$. Since $g(w)  \gq vu \gq uv$, we may assume that
$v = \alp_k$. The claims follow by induction, because elements of $\Z$, such as $\alp_k$ and $\bet_k$ commute with horocyclic elements.
\item If $g(w)$ is not the product of shorter horocyclic words, then we have $\bet_k = \alp_k = 0$ and either 
$g(w) = Tut$ or $g(w) = tuT$.  If $g(w) = tuT$, then we are done by induction.
So assume by contradiction $g(w) = Tut$.
\begin{enumerate}[i.)]
\item  Let $k \geq 2$ and $g(w) = Tut$. This is  impossible,
because $k \geq 2$, therefore $g(w) \gq Tg_1(u)t$, and $g_1(u)$ 
has $t$ as its first letter. 
\item Let $k=1$ and $g(w) = T\alp t$. Then we have 
$\alp \in q\Z$, so $g(w) = T\mu q t$ for some $\mu\in\Z$. But $T\mu q t \gq \mu p$ and 
$\abs {\mu p} < \abs {\mu q} +2$. 
Thus we find a contradiction again.
\end{enumerate}
\end{enumerate}
\end{proof}

\begin{example}\label{ex:expo}
Let $p=1$ and $q=2$. Our techniques will show that
$t^ka(Ta)^k$ is a \geo representation. The word is obviously 
horocyclic.
It is also clear that Lemma~\ref{danksprung} can be generalized such 
that in this case we find exactly $2^k$ \geo representations:
$$(\prod_{i=1}^k a^{\eps_i}t) \; a \; (\prod_{i=1}^k T a^{1-\eps_i})
\text{ where } \eps_i\in \oneset{0,1}.$$
\end{example}

The first phase in our computation to produce the \llex normal form for horocyclic elements will  a greedy linear time reduction to
compute the \llex normal form of \Breduced words where all 
$\theta_i$ are equal to $T$. 
For this purpose we introduce the following notion.
A \emph{slope} is a word $w$ which has the 
form $$w= \alp_0T\alp_1\cdots T\alp_k.
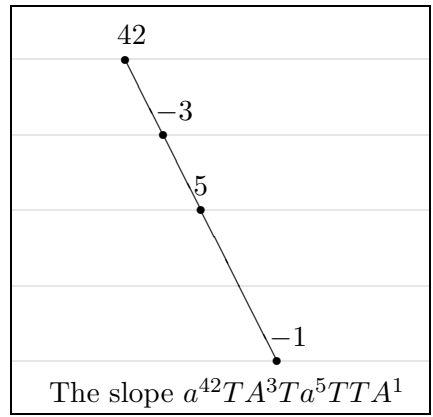
$$

\begin{center}
\begin{BS}{9}{4}
	\BSdesc{\text{The slope $a^{42}TA^{3}Ta^5TTA^{1}$}}
	\BSoff{2}{4}
	\BSa{42}\BST\BSa{-3}\BST\BSa{5}\BST\BST\BSa{-1}
\end{BS}
\end{center}

In \cite{elder09linear} slopes are called \emph{words of type N}. 

\begin{proposition}\label{fritzchen}
Let $w\in \oneset{t,T,a,A}^*$ be a horocyclic 
word 
with  $w \gq \alp \in \Z$. 
Then we can compute 
in linear time  a slope $u$ and a natural number $\ell \in \Theta(\log \abs \alp)$
such that 
$$\llnf{w} = t^\ell \llnf{u}$$ 
\end{proposition}

\begin{proof}
We begin with some precomputational steps.  
\begin{enumerate}[1.)]
\item  Replace 
$w$ by the corresponding word 
$u = \del_0\theta_1\del_1\cdots \theta_m\del_m$
in the free product $\FP$.
\item Compute 
a \Breduced 
equivalent word $u' = \del'_0\theta'_1\del'_1\cdots \theta'_{\ell'}\del_{\ell'}$
keeping integers in binary representation. 
If ${\ell'}\neq 0$, then $w$ was not horocyclic and we can stop.  
\item Now we may assume that $w= \alp\in \Z$.
\end{enumerate}

Next, we perform a greedy linear time algorithm. By symmetry we 
assume $0 \leq \alp\in \N$. 
The basic idea is that as long 
as $\alp \geq 2q$, the unary notation 
$a^\alp$  is not a {\llex}ly first representation,
because $ta^{2p}T<_{ll} a^{2q}$. This motivates  
rewriting  $\alp = q\mu + \bet$ with $0\leq \bet < q$. We can replace 
$\alp$ by $tp\mu T\bet$, and we have  $0 \leq  p\mu <  \alp$. 
Repeating this greedy process for $\mu p$ as long as 
$p \mu \geq 2q$ yields in linear time a representation 
$w\gq t^\ell\bet_0 T\bet_1 \cdots T \bet_\ell$ with $\ell \in \Theta(\log \abs \alp)$,
$0 \leq \bet_i < q$ for $i \neq 0$,  and $0 \leq  \bet_0 < 2q$.

Define $\alp_i= t^i \bet_0 T\bet_1 \cdots T \bet_i$. 
For each $i$ we found $\mu_i\ge 0$ such that 
$\alp_i= q\mu_i + \bet_i$ and
$$ t^i\bet_0 T\cdots  \bet_{i-1} T \gq p \mu_i.$$
We have 
$\alp = \alp_\ell$ and is enough to show that
the word 
$\llnf{\alp'}$ begins with $t^i$ for each $\alp'$ with $\alp_i \leq \alp'$. This is trivial for $i= 0$.
Now let $i \geq 1$. We know $2q\leq \alp_i \leq \alp'$, 
so by Lemma~\ref{danksprung} we obtain: 
$$\llnf{\alp'} = t^{i'} \gam_0 T \bet_1'\cdots T\bet_{i'}'$$
for some ${i'} \geq 1$. Moreover, $\abs{\bet_{i'}'} < q$.
We have
$\alp'\equiv \bet_{i'}' \bmod q$. Hence there exists $\mu'\ge 1$ with
$\alp' = \mu' q + \bet_{i'}'$
Recall that  $\alp_i = \mu_i q + \bet_i$ where $0\leq \bet_i< q$.
Therefore $\alp_i \leq \alp'$ implies $\mu_i \leq \mu'$. 
Induction applies, and we can conclude that $\llnf{p\mu'}$
begins with $t^{i-1}$. Thus, $\llnf{\alp'}$ begins with $t^{i}$.
\end{proof}

By Proposition~\ref{fritzchen} it is enough to 
compute \llex normal forms for slopes. Essentially, this can be performed by a finite transducer. Our algorithm relies on dynamic programming
and needs some preparation. 

Let $R$ denote the set (of constant size)
defined by  $$R=\set{\rho \in \Z}{\abs \rho \leq r},$$
where in this section the constant $r$ is the least 
positive integer such that
\begin{equation}\label{eq:fred}
r \geq  p \cdot \frac{r+q-1}{q} + q-1.
\end{equation}

Actually, minimizing the concrete value of $r$ 
is of little importance. All we need is 
that for all constants $c \geq  0$ there is a positive integer $r = r(c)$
such that 
\begin{equation}\label{eq:otto}
r \geq  p \cdot \frac{r+c}{q} +c.
\end{equation}

So, in Eq.~\ref{eq:fred} the constant $c$ was chosen to be $q-1$,
which guarantees in particular that $\abs{\bet} < q$ implies $\bet \in R$.

\begin{proposition}\label{fritz}
There is a linear time algorithm (where the number of 
bit-operations is linear, too) which  on input 
a slope $w= \bet_0 T\bet_1 \cdots T \bet_\ell$ with
\begin{equation}\label{eq:smallbetas}
\abs{\bet_i} < q\text{ for }i \neq 0\text{, and }\abs{ \bet_0} < 2q
\end{equation}
computes 
the \llex normal form $\llnf{w}$. Moreover, we have
$$\llnf{w}= t^k\bet T\alp_1 \cdots T \alp_{k+\ell}$$
with $k \in \Oh(1)$. 
\end{proposition}

Note that the slope $u$ given by Proposition~\ref{fritzchen} satisfies 
(\ref{eq:smallbetas}). Thus: 

\begin{corollary}\label{cor:horopolytime}
The \llexnf of horocyclic elements can be computed in linear time
(where the cost of arithmetic operations is considered to be constant). 
\end{corollary}

\begin{proof}{(Proposition~\ref{fritz})}
The input $w= \bet_0 T\bet_1 \cdots T \bet_\ell$ is \Breduced.

We initialize a table with entries $\llnf \rho$ for 
all $\rho \in \Z$ with $\abs\rho \leq r+ q$ 
in constant time. The desired form of the entries is guaranteed by Lemma~\ref{danksprung}.
For $\ell = 0$ we find the information in this table because
$\abs{\bet_0} < 2q$.
Now, let $\ell > 0$.

For each $0 \leq i \leq \ell$ and each $\gamma \in R$ 
we define the word $u(i, \gamma)$
by $$u(i, \gamma) =\bet_0 T \cdots \bet_{i-1}T \gamma.$$
We show that we can compute $\llnf{u(i, \gamma)}$ efficiently by table lookup. Recall that 
$\bet_\ell\in R$.

If $i=0$ then $u(0, \gamma)= \gam$ is a small integer, and 
$\llnf \gam$
is in the precomputed table.

Let $i \geq 0$ and assume for simplicity 
(and for a moment) that for all $\gamma \in R$ the values $\llnf{u(i,\gamma)}$
have been computed and stored in a table of size $\abs R$. 
We wish to compute the \llex normal form of 
$u(i+1, \rho_{i+1}) = \bet_0 T \cdots \bet_{i}T \rho_{i+1}$
for each $\rho_{i+1} \in R$. The word
$\llnf{u(i+1, \rho_{i+1})}$
ends in some $T\gam$, where 
$\abs \gamma < q$ and $\rho_{i+1} \equiv \gam \bmod q$, because 
a slope is \Breduced. Thus, 
$ \rho_{i+1} = \mu q + \gam$ for some $\mu \in \Z$ and we obtain $T \rho_{i+1} \gq \mu p T \gamma$. 
Now we have $\abs{\rho_{i+1}}\leq r$. 
Hence for $i=0$ we obtain
$$\abs{\mu p + \bet_i} 
\leq p \cdot \frac{r+q-1}{q} + 2q-1 \leq r + q.$$

For $i > 0$ we obtain
$$ \abs{\mu p + \bet_i}  \leq p \cdot \frac{r+q-1}{q} + q-1 \leq r.$$
 
Thus, in both cases we can compute  by table lookup:
$$\llnf{u(i+1, \rho_{i+1})}= \min\oneset{\llnf{u(i, \rho_i)}T\gamma},$$
where the minimum is taken over all $\rho_i \in R$ and $\abs \gamma < q$ such that 
$$ \exists \mu: \rho_i= \bet_i + \mu p \wedge \rho_{i+1} = \mu q + \gamma.$$

So far we have proven Proposition~\ref{fritz} except for the linear time bound. In fact, we have shown that the \llexnf of slopes (and thus of horocyclic elements) can be computed in polynomial time. 
It remains to explain how we manage to find the minimum in constant
time with a constant amount of information.  
Unfortunately, the explanation is slightly  technical. 
If the reader is interested in polynomial 
time results only, he or she is invited to skip the rest of the proof. 
An example of how the method works can be found in the appendix. 

For the explanation
we observe first 
that $$\abs{\;\abs{\llnf{u(i, \rho)}} - \abs{\llnf{u(i, \tau)}}\; } \leq 2r$$ 
for all $\rho, \tau \in R$. One idea is therefore that 
instead of keeping the words $\llnf{u(i, \rho)}$ in the table, we store in it 
only certain suffixes of these words of maximal length $2r$. 

This would be even more evident if it were enough to compute the 
\geo length. Then the table would just need to store the length differences
$$\abs{g(u(i, \rho)) } - \abs{g(u(i,0)) } \in \oneset{-2r, \ldots, 2r}$$
in addition to the absolute \geo length of one of these elements, say, $\abs{g(u(i,0))}$. 

As we are more ambitious than simply computing geodesic lengths, we need more subtle data structures. 
More precisely, factorize each $\llnf{u(i, \rho)}$ 
as 
$$\llnf{u(i, \rho)} = p_{i,\rho}\cdot s_{i,\rho}$$
such that $\abs{p_{i,\rho}} = \min\{\abs{\llnf{u(i,\tau)}}:\tau\in R\}$,
thus having $\abs{s_{i,\rho}}\leq 2r$
for all $\rho\in R$.
We modify the above algorithm so that after the $i$-th round we only the store the following information:
\begin{itemize}
\item A list of the $s_{i,\rho}$ for $\rho\in R$, and
\item the lexicographical ordering of the prefixes $p_{i,\rho}$ ($\rho\in R$).
\end{itemize}
Note that this information only takes a constant amount of space
and that all $p_{i,\rho}$ have the same length. 

It is clear now that we can perform the minimum search in constant
time. Say, we need to compare $\llnf{u(i,\rho)}T\gam$ and $\llnf{u(i,\tau)}T\del$. 
As $\abs{p_{i,\rho}} = \abs{p_{i,\tau}}$ we can compare
the length by looking at the lengths of $s_{i,\rho}T\gam$ and $s_{i,\tau}T\del$.
In case $\abs{s_{i,\rho}T\gam} = \abs{s_{i,\tau}T\del}$
we check the ordering between $p_{i,\rho}$ and $p_{i,\tau}$.
If $p_{i,\rho}=p_{i,\tau}$, we finally compare $s_{i,\rho}T\gam$ and $s_{i,\tau}T\del$. 

The length of the shortest word $\llnf{u(i, \rho)}$
(with $\rho \in R$) 
strictly increases from round $i$ to round $i+1$, 
since a factor of type $T\gam$ is concatenated. So we can actually update the 
information in constant time. 
We simply have to compute new suffixes and a new linear order on $R$. As we 
need to recover $p_{i+1,\rho}$ later, the algorithm outputs a column vector with
the portions of the suffixes which were cut due to length constraints
plus a pointer to the index of the preceeding column where the minimum was achieved. 

The output of the algorithm can be viewed as a matrix with 
$2r+1$ rows and a linear number of columns.
The entry $(\rho,i)$ contains a word of constant length and
a pointer to some $(\tau,i-1)$.
In the last round $\ell$ the output consists of 
the suffixes $s_{\ell,\rho}$ and, using these,
the \llex normal form 
for each $u(\ell,\rho)$ can be read 
by a single scan from right to left in the
matrix, starting at position $(\rho, \ell)$ and following the pointers. In particular, we can read
off $\llnf{w} = \llnf{u(\ell, \bet_\ell)}$
in linear time.
\end{proof}

\begin{corollary}\label{wiesiehtsieaus}
Let $w\in \oneset{t,T,a,A}^*$ be a horocyclic 
word 
with  $w \gq \alp \in \Z$ and let $\ell\in \N$ be the number
computed in Proposition~\ref{fritzchen}. 
Then we have 
$$\llnf{w}= t^k\bet T\alp_1 \cdots T \alp_k$$
with $0\le k - \ell \in \Oh(1)$,  
$\abs {\alp_i} < q$, $\alp_k \equiv \alp \bmod q$, $\abs {\bet} < 2q$, and if $k > 0$ then 
 $\bet \in p\Z$.
\end{corollary}

\begin{proof}
Propositions~\ref{fritzchen} and~\ref{fritz} show
that $\llnf{w}$ has the form $t^k\bet T\alp_1 \cdots T \alp_k$
with $k - \ell \in \Oh(1)$. 
The other assertions are a direct consequence: $\abs {\alp_i} < q$, $\alp_k \equiv \alp \bmod q$, $\abs {\bet} < 2q$, and if $k > 0$, then 
$\bet \in p\Z$.
\end{proof}
 
Let us call a slope $w= \bet_0 T\bet_1 \cdots T \bet_\ell$ 
\emph{\sh,}  if there is a number $k$ such that 
$t^k u$ is horocyclic. Clearly, if such a $k$ exists, then we have 
$k = \ell$, so $k$ is unique. 
 
\begin{corollarybr}\label{allesreg}
\begin{enumerate}[1.)]
\item\label{ra}  The set of slopes in \llex normal form 
$$ \set{w = \bet_0 T\bet_1 \cdots T \bet_\ell \in \oneset{T,a,A}^*}{
w = \llnf{w}}$$
is regular. 
\item\label{rb}  If $p$ devides $q$, then the set of \sh slopes in \llex normal form
$$ \set
{w 
\in \oneset{T,a,A}^*}
{w = \llnf{w} \text{ and } w \text{ is \sh}}$$
is regular. 
\item \label{rc}  If $p$ devides $q$, then the set of 
horocyclic elements in \llex normal form 
$$ \set{w 
\in \oneset{T,a,A}^*}{
w = \llnf{w} \text{ and } w \gq \alp \in \Z}$$
is a deterministic (and unambiguous) linear context-free (one-counter) language; and it can be 
recognized in $\log$-space. The growth series of the horocyclic subgroup is a rational 
function. 
\end{enumerate}
\end{corollarybr}

\begin{proof}
Formally, Statement \ref{ra} does not follow from assertion in Proposition~\ref{fritz}, but analyzing its proof shows that all arithmetic computations 
concern only a constant number of integers of constant size. This can be 
done in the finite control. For accepting a \llex normal form, we do not need any output, we just have to check that the input agrees locally 
with a potential output. Again, this can be done in the finite control.
The result follows. 
 
Statement \ref{rb} follows from \ref{ra} because a test 
whether a slope is \sh can be done by counting modulo $p$. 
 
Statement \ref{rc} follows from \ref{rb}, we simply have
to check additionally that the number of $t$ and $T$ match. 
This can be done by a deterministic (one-counter) pushdown automaton, which 
translates to an unambiguous linear \confree grammar. The growth series of
unambiguous linear context-free languages is rational, \cite{MR42:4343}.
\end{proof}

\begin{remark}\label{rem:ft}
Statement \ref{rc} is essentially a result due to Freden et al., see \cite[Prop.~9.1]{fcs09}. 
Statement \ref{ra} is slightly more general than \cite[Prop.~9.1]{fcs09}
because our result holds for all $p$ and $q$. There is a crucial difference
for $p \mid q$.
It is only when $p \mid q$ that we can test with a push-down automaton whether an input word $w= t^\ell\bet_0 T\bet_1 \cdots T \bet_\ell$
is horocyclic, see \cite[Thm.~7.2]{fcs09}.
So we need this as a \emph{promise} in order to produce horocyclic words with the sort of finite transducer we used in the proof of Proposition~\ref{fritz}.
A transducer cannot compare  $m= \ell$; and 
even if $m= \ell$, it cannot test whether the input is horocyclic. 
This part is however trivial for  $p \mid q$ by counting modulo $p$. 
This is why \ref{rb} follows from \ref{ra} easily in this case. 
\end{remark}

\section{Peak normal forms}\label{sec:pnf}

We consider $\Delta = \Z \cup \oneset{t,T}$
as an infinite alphabet of symbols with the 
following linear order:
\begin{enumerate}[1.)]
\item  $t < T < \alp$
\item  $\alp < \bet$, if $\abs \alp < \abs \bet$
\item  $\alp \leq  \bet$, if $0 \leq \alp = \abs \bet$
\end{enumerate}
The \llex order on $\Delta^+$ transfers to a linear 
order $\lldel$ on words of $\FP\subseteq\{a,A,t,T\}^\ast$. 

Consider a word $w \in \FP$ such that 
\begin{equation*}
	w  = \alpha_0'\theta_1'\alpha_1'\cdots\theta_\ell'\alpha_\ell'.
\end{equation*}

For every position $i\in\oneset{0,\ldots,\ell}$ of $w$ we define its height by $h(0)=0$ and $h(i)=h(i-1)+1$ if $\theta_i' \sq 
t$ or $h(i)=h(i-1)-1$ when $\theta_i' \sq T$. The height of $w$ is defined by 
$h(w)=\max\set{h(i)}{0\le i\le \ell}$. 

Again, this notion arises from the shape of the path of $u$, 
if reading a $T$ means going downwards and
reading a $t$ upwards, as we have done before. Consider e.g.
$w=ATa^2taTattA^{10}TaTa^2tta^{42}T$:

\begin{center}
\begin{BS}{10}{2}
	\BSdesc{\text{A word of height 1
	}}
	\BSoff{0}{1}
	\BSheights{1}
	\BSa{-1}\BST\BSa{2}\BSt\BSa{1}\BST\BSa{1}\BSt\BSt\BSa{-10}\BST\BSa{1}\BST\BSa{2}\BSt\BSt\BSa{42}\BST\BSa{}
\end{BS}
\end{center}

For the peak normal form we compute some \Breduction first. 
Thus, we let $w \gq u$ and 
\begin{equation*}
	w \gq u = \alpha_0\theta_1\alpha_1\cdots\theta_k\alpha_k
\end{equation*}
where $u$ is \Breduced. 
Recall that the sequence $(\theta_1, \ldots, \theta_k)$ depends  on $w$ only.

We say that the position $i$ is the \emph{peak} of $u$, 
if $i$ is maximal among all $i$ with $h(i)= h(u)$, i.e. it is the rightmost among the hightest local peaks. 
Define $u_1$ by 
\begin{equation*}
	u_1= \alpha_0\theta_1\cdots\alpha_{i-1}\theta_{i-1}
\end{equation*}
We obtain a natural factorization
$$u = u_1 \alp_i u_2.$$

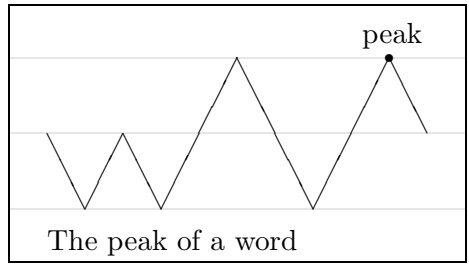
\begin{center}
\begin{BS}{10}{2}
	\BSdesc{\text{The peak of a word}}
	\BSoff{0}{1}
	\BST\BSt\BST\BSt\BSt\BST\BST\BSt\BSt\BSa{\!\!\!\!\text{peak}}\BST
\end{BS}
\end{center}

We say that a \Breduced word $u = \alpha_0\theta_1\alpha_1\cdots\theta_k\alpha_k$ is the 
\emph{\Bpnf}of 
$w$ if 
\begin{enumerate}[1.)]
\item  $w \gq u$ and $\abs{g(w)} = \Abs{u}$
\item  Among all choices satisfying $w \gq u$ and $\abs{g(w)} = \Abs{u}$
we choose the one where $u_1$ is the first in the order $\lldel$,
after that we minimize  $\ov u_2$ in the order $\lldel$.
\end{enumerate}

We say that a word $v$ is the 
\emph{\peaknf}of  
 $w$
if 
$$v = \llnf{\alpha_0}\theta_1\llnf{\alpha_1}\cdots\theta_k\llnf{\alpha_k}$$
where $u = \alpha_0\theta_1\alpha_1\cdots\theta_k\alpha_k$ 
is in \Bpnf. 
The \peaknf of a word $w$ is geodesic, and it is denoted 
as $\pnf{w}$ in what follows.

\section{Difficult cases and solution for hills}\label{sec:di}

Let $u = \alp_0\theta_1\alp_1\cdots \theta_k\alp_k$ be \Breduced. Due to Section~\ref{sec:ghw} we can compute the norm $\Abs{u}$ and 
\llex normal forms for each $\alp_i$. Unfortunately,  we are still
not able to compute a geodesic for $u$ efficiently, in general. 
But at least we can identify difficult cases and solve the the problem for so-called hills. 
A \emph{hill} is a word $w$ such that we 
have $$w \gq u = \bet_{\ell}t \cdots \bet_1 t\alp_0T\alp_1\cdots T\alp_m$$
for some $u$ and $\ell,m \geq 0$. 

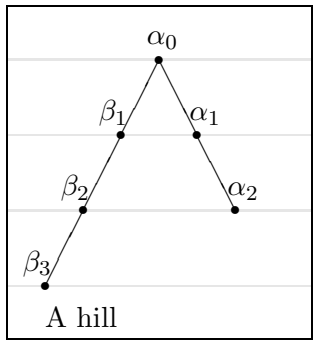
\begin{center}
\begin{BS}{6}{3}
	\BSdesc{\text{A hill}}
	\BSoff{0}{0}
	\BSa{\!\!\!\beta_3}\BSt\BSa{\!\!\!\beta_2}\BSt\BSa{\!\!\!\beta_1}\BSt
	\BSa{\!\alpha_0}
	\BST\BSa{\alpha_1}\BST\BSa{\alpha_2}
\end{BS}
\end{center}

Note that all horocyclic 
words are hills, so their \llex normal form is a hill representation. 
It is clear that $w$ is a hill \IFF 
its \Breduction already has the form
$\bet_{\ell}t \cdots \bet_1 t\alp_0T\alp_1\cdots T\alp_m$. 

We show that we can compute the \peaknf for
hills $w$ very efficiently, see also \cite{elder09linear}. If $w$ is not a hill, then 
we can in linear time reduce the computation of the 
\peaknf of $w$ to the the computation of the 
\peaknf of so-called \di words, which are defined below. 

For $p\mid q$ we will solve the remaining difficult cases
in the next section.  

In this paper we call a word $w$ \emph{\di,}
if its \Breduction is $\alp_0\theta_1\alp_1\cdots \theta_k\alp_k$ with  $\theta_1= T$ and $\theta_k= t$, in other words it has the form $\alp T v t \bet$.

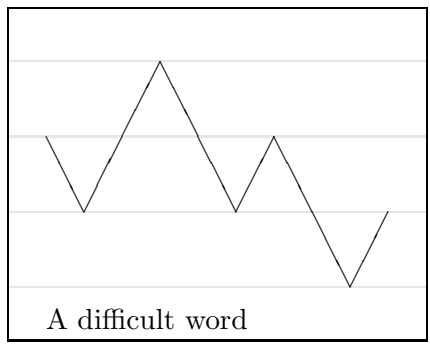
\begin{center}
\begin{BS}{9}{3}
	\BSdesc{\text{A difficult word}}
	\BSoff{0}{2}
	\BST\BSt\BSt\BST\BST\BSt\BST\BST\BSt
\end{BS}
\end{center}

Note that every \Breduced word $u$ has a unique representation 
$$u = \alp_1t\cdots \alp_kt D T \bet_m \cdots T \bet_1$$
where either $D = \delta$ is horocyclic (making $u$ a hill) or $D=\alp T v t \bet$, i.e. $D$ is \di. 

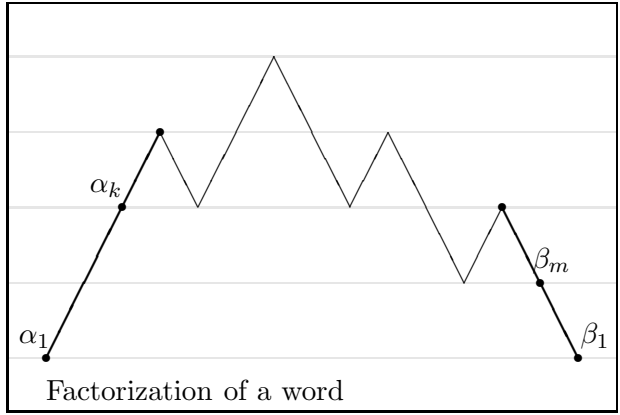
\begin{center}
\begin{BS}{14}{4}
	\BSdesc{\text{Factorization of a word}}
	\BSoff{0}{0}
	\thicklines
	\BSa{\!\!\!\!\alp_1}\BSt\BSt\BSa{\!\!\!\!\!\alp_k}\BSt\BSa{}
	\thinlines
	\BST\BSt\BSt\BST\BST\BSt\BST\BST\BSt
	\thicklines
	\BSa{}\BST\BSa{\bet_m}\BST\BSa{\ \bet_1}
\end{BS}
\end{center}

Assume that we have an algorithm that computes {\peaknf}s for \di words. 
Let $\T(d)$ denote the maximal time to compute $\pnf{D}$ for $\Abs{D} = d$ for $\Abs{D} = d$.
We assume that $\T(d+c)\in\Oh(\T(d))$ if $c$ is a constant.
We will show later that this assumption is justified in the case $p\mid q$. It remains
reasonable in general, since it holds for every sensible complexity bound, such as polynomials or singly exponential functions.

The following result generalizes Proposition~\ref{fritz}.
For its proof we use the constant $r$
as defined in Eq.~\ref{eq:fred}.

\begin{theorem}\label{thm:otto}
Let $u = \alp_kt\cdots \alp_1t D T \bet_1 \cdots T \bet_m$ be \Breduced 
with $D$ either horocyclic ($D\in \Z$) or $D$ difficult.

\begin{enumerate}[1.)]
\item If $D$ is horocyclic, then we can compute the \peaknf $\pnf{u}$ in linear time. 
\item If $D$ is \di, then we can compute the \peaknf $\pnf{u}$ in 
$$\Oh(\T(\Abs D + \max\{\Abs{\alp_i},\Abs{\bet_j}\}) + \Abs{u}).$$
\end{enumerate}
\end{theorem}

\begin{proof}
Applying a greedy algorithm similar to the one used in the proof of Proposition~\ref{fritzchen}
we can ensure that $0\le\abs{\alp_i},\abs{\bet_j}<q$ for all $1\le i\le k$, $1\le j\le m$. 
Doing this, $D$ might become longer, but this doesn't affect the case \textit{1.)} since we have a linear time algorithm for horocyclic words. In case \textit{2.)} it suffices to proove
a time bound of $\Oh(\T(\Abs{D})+\Abs{u})$. 

Now let
$$u(i,j,\rho,\delta)=\rho t\alp_{i-1}t\cdots \alp_1t D T \bet_1 \cdots T \bet_{j-1}T\delta.$$
By induction on $i+j$ we prove that the theorem holds for every word $u(i,j,\rho,\delta)$ with $\rho,\delta\in R=\set{\gam \in \Z}{\abs \gam \leq r}$. 

If $i+j=0$, we need to compute the \peaknf of $\rho D\delta$, a word of  length at most $\abs{D}+2r$. If $D$ is horocyclic, Corollary~\ref{cor:horopolytime} tells us that this can be done in linear time. If $D$ is \di, then we need $\Oh(\T(\Abs{D}+2r))$ which, by our assumption on $\T$, is bounded by $\Oh(\T(\Abs{D}))$. 

Now consider $i+j>0$. By symmetry we may assume that $i>0$. Since 
$D$ and $u(i,j,\rho,\delta)$ are \Breduced, $\pnf{u(i,j,\rho,\delta)}$ starts with $\gam t$ and we have $\gam\equiv\rho\bmod q$. In addition, $\abs{\gam}<q$, since
$\abs{a^qt} = \abs{A^q t}= q+1 > p+1=\abs{ta^p} = \abs{tA^p}$.
So, since the peak is inside of $D$, like in the horocyclic case, we have
$$\pnf{u(i,j,\rho,\delta)}=\min\oneset{\gam t \pnf{u(i-1,j,\rho',\del)}}$$
where the minimum is taken over all $\gam$ and $\rho^\prime$ such that $\abs{\gam}<q$ and $\rho=\gam+\mu q$ and $\rho^\prime=\mu p+\alpha_{i-1}$. Again, since $\rho\in R$, we have $\rho^\prime\in R$, so induction applies. 

For the implementation we use again a table of size ${\abs R}$.
It stores the length information 
$$\abs{\pnf{u(i,j,\rho,\delta)} } - \abs{ \pnf{u(i,j,0,\delta)}} \in \oneset{-2r, \ldots, 2r}$$
as well as the ordering between various words $\pnf{u(i,j,\rho,\delta)}$ and 
$\pnf{u(i,j,\tau,\delta)}$ as a linear order on $R$. 

This allows the minimum search in constant time by table lookup 
in a table of constant size with constant size entries. 
Indeed the ordering between $\gam t \pnf{u(i,j,\rho,\delta)}$
and  $\del t \pnf{u(i,j,\tau,\delta)}$ is dominated by the length
and next by the ordering between $\gam t $ and $\del t$.
For $\gam t = \del t$ we can refer to the ordering between 
 $\pnf{u(i,j,\rho,\delta)}$
and  $\pnf{u(i,j,\tau,\delta)}$.

The table update is possible in constant time, too.
The output is produced from right to left in this phase. 
\end{proof}

\begin{corollary}\label{cor:hill}
Let $w$ be a hill. Then we can compute the \peaknf $\pnf{w}$ 
in linear time. 
\end{corollary}

The proof of the following Corollary~\ref{peaknfhill} is rather technical.
We do not use the result anywhere, so we leave its proof to the 
interested reader. 
\begin{corollary}\label{peaknfhill}
Let $p$ divide $q$. Then 
the set of hills in \peaknf 
$$\set{w}{w= \pnf{w} \text{ and $w$ is a hill}}$$
is a deterministic (and unambiguous) context-free language. Its growth series 
is a rational function. 
\end{corollary}


\section{Complete solution when $p$ divides $q$}\label{sec:pq}
\renewcommand{\t}{t}
\renewcommand{\T}{{T}}
From now on we assume that $p$ divides $q$. We give an algorithm for computing geodesics of any element in $BS(p,q)$ which runs in quadratic time, assuming that arithmetic operations in $\mathbb{Z}$ take constant time. One reason why things simplify is that we have 
a canonical mapping
$$\pi: \BSpq \to (\Z/ p \Z) * \oneset{t,T}^*$$ 
which is induced by mapping a 
\Breduced word $\alpha_0\theta_1\alpha_1\cdots\theta_k\alpha_k$ 
to:
\begin{equation*}
(\alpha_0 \bmod p) \theta_1(\alpha_1 \bmod p)\cdots\theta_k(\alpha_k \bmod p).
\end{equation*}
Using the confluent string rewriting system $BS$, we see that, in particular,  the sequence 
$(\alpha_0 \bmod p, \ldots, \alpha_k \bmod p) \in (\Z/ p \Z)^k$ depends on the image of 
$u$ in $\BSpq$, only. 

Recall that according to Eq.~\ref{eq:lp} the number of
sinks is denoted by $s(w)$. We have $1 \leq s(w) \leq 1 + \Abs{w}$.
Theorem~\ref{thm:main} is the main result of the paper. 

\begin{theorem}\label{thm:main}
Let $p$ be a divisor of $q$.
Let $w \in \oneset{t,T,a,A}^*$ be an input word. 
Then we can compute a geodesic and its geodesic length in 
quadratic time $\Oh(s(w)(s(w) + \Abs{w}))$.
\end{theorem}

The rest of the paper is devoted to proving this result. 
The formal proof of Theorem~\ref{thm:main} is postponed to Section~\ref{thisistheendmyfriend}

We may start with a word $u \in \FP$ such that 
\begin{equation*}
	w \gq u = \alpha_0\theta_1\alpha_1\cdots\theta_k\alpha_k.
\end{equation*}

Recall the definition of the height from Section~\ref{sec:pnf}. 
The word $w$ is called a \emph{valley}, if $h(u)=0$ and $h(k)=0$.

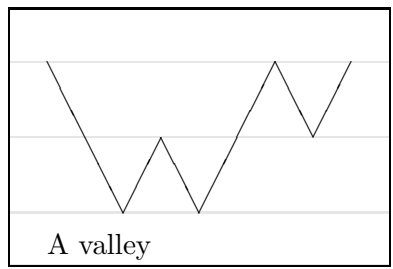
\begin{center}
\begin{BS}{8}{2}
	\BSdesc{\text{A valley}}
	\BSoff{0}{2}
	\BST\BST\BSt\BST\BSt\BSt\BST\BSt
\end{BS}
\end{center}

A \Breduction cannot increase 
the height and leaves the height of the last position 
invariant. Therefore, whether or not a word is a valley can be checked
on its \Breduction. 

Valleys are generated by the following context-free grammar where $S$ is an  axiom: 
\begin{equation*}
	S\to SS\ \vert\ \alpha T S\beta t\ \vert\ \alpha
\end{equation*}

The following lemma can be based on this grammar. The proof is  
straightforward, 
but it uses in a crucial way that $p \mid q$. 

\begin{lemma}\label{comval} Let $ p \mid q$. 
	Let $v$ be a valley. Then $pv\gq vp$.
\end{lemma}

\begin{proof}
 We have $p+\alp = \alp + p $ in $\Z$. If $u = vw$ and $v$, $w$ are 
 valleys, then, by induction, $pvw\gq vpw \gq vwp$. 
 If $v = \alpha T w\beta t$, then
 \begin{align*}
 pv& = (p+\alpha) T w\beta t
 \\ & \gq 
  \alpha T qw\beta t\\ 
  & \gq  \alpha T w(\beta+ q) t &\mbox{by induction, because $p \mid q$}\\
  & \gq  \alpha T w\beta tp.
  \end{align*}
\end{proof}

\subsection{Reduction to valleys}\label{sec:rtov}

Remember that we assume $ p \mid q$ and it  only remains to compute 
 \peaknf{}s 
for \di 
\Breduced words 
$u = \alp T v t \bet$. Let $\ell = h(u)$. There is a unique $m \in \N$ such that $T^\ell u t^m$ is a valley.
The word  $T^\ell u t^m$ is still \Breduced!
By Corollary~\ref{gunnar} we can compute 
the \peaknf  $\pnf{T^\ell u t^m}$ in quadratic time. 
The word $\pnf{T^\ell u t^m}$ has the form 
$\alp_1 T \cdots \alp_\ell T w$. 
Now, the sequence $\pi(T^\ell u)$ begins with 
$T^\ell \in (\Z/ p\Z)\ast\{t,T\}^\ast$. 
In other words, we obtain 
$\alp_1 \equiv  \cdots  \equiv  \alp_\ell  \equiv 0 \bmod p$.
Each of the first $\ell$ positions  finds a  matching position
greater than $\ell$ of the same height. So, we can shift all
integers $\alp_i$ for $1 \leq i \leq \ell$ to the right by Lemma~\ref{comval} without  increasing the length. 
Hence, the \peaknf  begins actually with 
$T^\ell$. For $m=0$ we conclude that $\pnf{T^\ell u t^m} = 
T^\ell \pnf{ u t^m} = T^\ell \pnf{ u}$ and we are done. 

It remains to compute $\pnf{u}$ for $m> 0$. 
In this case we write 
$$\pnf{T^\ell u t^m} = u_1\gamma u_2= 
T^\ell v_1\gamma u_2$$
where $\gamma$ is the integer on the (rightmost) peak of $T^\ell u$, which is also the (rightmost) peak of $u$. The peak of $u$ is not the last position in $u$, because $m>0$.
Hence,  $u_2$ is a valley which begins with 
the letter $T$. Moreover, we may assume that either  $u_1= 0$
or  it ends with $t$. 

Now we concentrate on the valley  $\gamma u_2$.
This time we compute, again in quadratic time 
the \peaknf \emph{from right to left}. This means we compute
$$\ov{\pnf{\ov  {\gamma u_2}}} = \delta v_2 t^m,$$
where $v_2$ begins with $T$. 


We claim  that $\pnf{u} =  v_1 \delta v_2$. Indeed, 
let  $\pnf{u} =  v_1' \delta' v_2'$ where $v_2'$ begins with $T$, 
$\pi(v_2) = \pi(v_2')$ and either $v_1= 0$ 
or $v_1$ ends in $t$. 

Obviously,  ${v_1 \delta v_2} \gq {v_1' \delta' v_2'}$ and 
$\abs{v_1 \delta v_2} = \abs{v_1' \delta' v_2'}$.
Now, $v_1' \lldel v_1$, hence  $T^\ell v_1' \lldel T^\ell v_1$
and therefore $v_1'= v_1$.

As we have $\abs{v_1 \delta v_2} = \abs{v_1' \delta' v_2'}$ and $v_1'= v_1$ we must have $\ov{v_2'} \lldel \ov{v_2}$ and 
$\ov{v_2't^m} \lldel \ov{v_2t^m}$.
Now, $\delta v_2 \gq \delta' v_2'$ and hence 
$\gamma u_2 = \delta v_2 t^m \gq \delta' v_2' t^m$.
We conclude $\delta v_2 = \delta' v_2'$, and finally 
$\delta = \delta'$. 

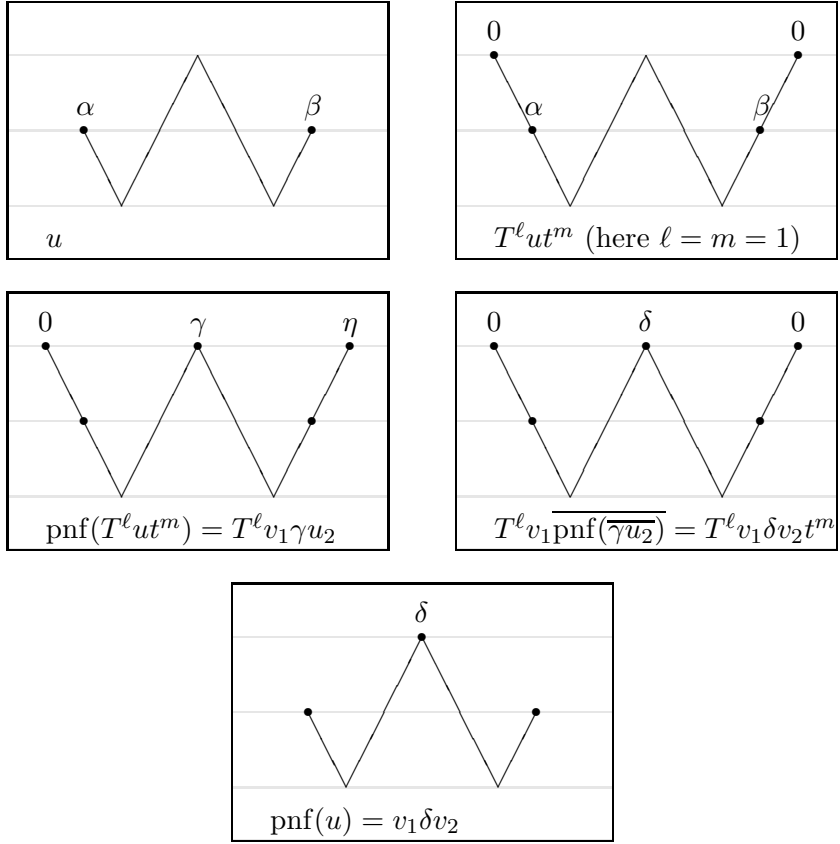
\begin{center}
\begin{BS}{8}{2}
	\BSdesc{$u$}
	\BSoff{1}{1}
	\BSa{\alpha}\BST
	\BSt\BSt\BST\BST\BSt\BSa{\beta}
\end{BS}
\begin{BS}{8}{2}
	\BSdesc{$T^\ell ut^m\ (\text{here }\ell=m=1)$}
	\BSoff{0}{2}
	\BSa{0}\BST\BSa{\alpha}\BST
	\BSt\BSt\BST\BST\BSt\BSa{\beta}\BSt\BSa{0}
\end{BS}
\end{center}

\begin{center}
\begin{BS}{8}{2}
	\BSdesc{$\pnf{T^\ell ut^m} = T^\ell v_1\gamma u_2$}
	\BSoff{0}{2}
	\BSa{0}\BST\BSa{}\BST
	\BSt\BSt\BSa{\gamma}\BST\BST\BSt\BSa{}\BSt\BSa{\eta}
\end{BS}
\begin{BS}{8}{2}
	\BSdesc{$T^\ell v_1 \ov{\pnf{\ov{ \gamma u_2}}} =T^\ell v_1 \delta v_2 t^m $}
	\BSoff{0}{2}
	\BSa{0}\BST\BSa{}\BST
	\BSt\BSt\BSa{\delta}\BST\BST\BSt\BSa{}\BSt\BSa{0}
\end{BS}
\end{center}

\begin{center}
\begin{BS}{8}{2}
	\BSdesc{$\pnf{u} =v_1 \delta v_2 $}
	\BSoff{1}{1}
	\BSa{}\BST
	\BSt\BSt\BSa{\delta}\BST\BST\BSt\BSa{} 
\end{BS}
\end{center}

\subsection{Standard valleys}\label{sec:sv} 
A word $V$ is called a \emph{standard valley},
if $V = \alpha_1\theta_1\cdots\alpha_k \theta_k\bet$ is a \Breduced
valley such that $\bet = 0$, $\abs{\alp_i} < q$ for all $1 \leq i \leq k$, 
and $\abs{\alp_i} < p$ whenever 
$1 \leq i \leq k$ and $\theta_i = T$.

Note that non-trivial \sv{}s end in the letter $t$.
By Section~\ref{sec:rtov} we may assume that we start the computation of
a \peaknf with  a \Breduced valley. The next lemma is a 
key step for the sequel. 

\begin{lemma}\label{lem:hugo}
Let $ p \mid q$ and $v$ be \Breduced valley, then we find in linear time
a \sv $V$ and an integer $\gamma\in \Z$ such that 
$v \gq V \gamma$. 
\end{lemma}

\begin{proof}
First note that $V \gamma$ is always \Breduced,
because a non-trival \sv ends in the 
letter $t$.   
\begin{enumerate}[1.)]
\item   For  $v = \alp \in \Z$ we can choose $V=0$.
\item  Let $v= uw$
where $u$, $w$ are shorter valleys. By induction, there are 
a \sv $U$ and $\bet\in \Z$ such that
$u \gq U \bet$ and $U \bet$ is \Breduced.  Hence 
$v \gq U\bet w $  and  $ \bet w$ is a \Breduced
valley. By induction, there are a \sv $W$ and $\gamma\in \Z$ such that
$v \gq U \bet w \gq UW \gamma$.
\item Let $v = \alp T u \bet t$. Then $u$ is a \Breduced valley;
by induction,   there are 
a \sv $U$ and $\delta\in \Z$ such that
$u \gq U \delta$. 
Write $\alp = \mu p + \alp'$ with $0 \leq \alp' < p$, 
then we have  
$$v \gq \alp' T U (\delta + \mu q +\bet) t.$$

Write $\delta + \mu q +\bet = \nu q + \bet'$ with $0 \leq \bet' < q$, 
then we have  
$$v \gq \alp' T U \bet' t \nu p.$$ The word $V =\alp' T U \bet' t$ 
is a \sv. 
\end{enumerate}
\end{proof}

By Lemma~\ref{lem:hugo} it is enough to compute \geo normal
forms for words $u = Vw$ where $V$ is a \sv and $w \gq \gamma$ is
horocyclic. Again, this can be done with a dynamic programming approach as we demonstrate in the next subsection. 

\subsection{Computing geodesics for $Vw$}\label{sec:csv}
In this section $V$ is always a \sv and $w$ is word such that 
$tw$ is \Breduced. 

Every  standard valley can be generated by the 
following context-free grammar:
\begin{equation*}
	S\to SS\ \vert\ \alpha T S\beta t\ \vert\ 0\text.
\end{equation*}

Here $S$ is the axiom and
$\alp$, $\bet$ denote intergers which are also viewed as terminal symbols and satisfy $\abs{\alp} < p$ and $\abs{\bet} < q$.
Actually, the grammar produces non-\sv as well, because we can 
produce valleys which are not \Breduced. But this is not important
for what follows. 

The number of sinks (c.f. Eq.~\ref{eq:lp}) of a standard valley $V$ 
admits a nice recursion: 
\begin{equation*}
	s(V)=\begin{cases}
	1&\text{if }V=0 \cr
	s(U)+s(W)&\text{if }V=UW \text{ and }   U \neq 0 \neq W\cr 
	s(U)&\text{if }V=\alpha T U\t\bet.  
	\end{cases}
\end{equation*}

We now choose a constant $r$ similar as in Eq.~\ref{eq:otto}, but this time using
the constant $c=3q-2$. 
More precisely, we let $r\in \N$ be minimal such that 
\begin{equation}\label{const}
r \geq  p \cdot \frac{r+3q-2}{q}.
\end{equation}

The next step is to define for each \sv $V$ 
a \emph{range} $R(V)\subseteq p\Z$ such that $\abs{R(V)} \in \Oh(s(V))$.
The precise definition and necessary properties of $R(V)$ are determined 
by the next rather technical lemma. 

\begin{lemma}\label{lem:range}
Let $p \mid q$ and let $r$ be the constant of Eq.~\ref{const}. 
Given a \sv $V$ as input, we can (define and) compute in linear time 
a range $R(V)$ such that the following two properties hold: 
\begin{enumerate}[1.)]
\item  $R(V) \subseteq \set{\rho \in p \Z}{ \abs \rho \leq r \cdot s(V)}$. 
\item For all $\rho \in R(V)$ there is a \sv $V_\rho$ 
with $V \gq V_\rho \rho$. 
\end{enumerate}
\end{lemma}

\begin{proof}
 
\begin{enumerate}[1.)]
\item For $V=0$ define $R(0)= \oneset{0}$.
\item Let $V=UW$ and $ U \neq 0 \neq W$.
We define 
$$R(V) = \set{\rho \in  \Z}{ \rho = \sigma + \tau, \sigma \in R(U),
\tau \in R(W)} \subseteq p \Z.$$
Clearly, $\rho\in R(V)$ implies 
$\abs \rho \leq  r \cdot (s(U) + s(W)) = r \cdot s(V)$. 

Let $\sigma \in R(U)$ and $\tau \in R(W)$.
As $\sigma \in p \Z$ we can write 
$$V \gq U_\sigma \sigma W \gq U_\sigma W_\tau(\sigma +\tau).$$ 
\item Let $V = \alp T U \bet t$. 
For $\sigma \in R(U)$ consider in a first step all values 
$\alp'$ with $\alp = \eps p + \alp'$ and $\abs{\alp'} < p$.
We know $\abs \eps \leq 1$. 
We can write 
$$V \gq \alp' T U_\sigma{(\sigma + \eps q+ \bet) t}.$$
Note that $\abs{\sigma + \eps q+ \bet}\leq 
r\cdot s(V) + 2q -1$. 
Thus $\sigma$ leads to some values $\bet'$ and hence some values $\rho$ 
such that
$\sigma + \eps q+ \bet = \mu q + \bet'$ and 
$\abs{ \bet'} < q$ and $\rho = \mu p$. 
Note that 
$$\abs \rho = \abs{\mu} p \leq p \cdot \frac{r\cdot s(V) + 3q -2}{q}
\leq r\cdot s(V).$$
We define $R(V)$ to be the set of all $\rho$ which are possible outcomes for some $\sigma \in R(V)$.

Of course, we can write 
$$V \gq \alp' T U_\sigma{(\sigma + \eps q+ \bet) t}
\gq \alp' T U_\sigma \bet' t \rho.$$
\end{enumerate}
\end{proof}

\begin{theorem}\label{whou} Let $ p \mid q$.
We can design a quadratic time algorithm running in time $\Oh(s(V)({\Abs V}+s(V)))$ which solves the following problem.

{\textbf {Input:} }A \sv $V$

{\textbf {Problem:} }Compute for all $\rho \in R(V)$  words
$V_\rho$ in \peaknf such that the following condition is satisfied. 

For all $w$, where $tw$ is \Breduced, it holds: 
\begin{equation}\label{nuwirdsgut}
\pnf{Vw}= \min\set{V_\rho \pnf{\rho w}}{\rho \in R(V)}
\end{equation}
\end{theorem}

\begin{proof}
Note that for \sv{}s the \llex ordering coincides 
with ordering in \peaknf. Thus, if we wish 
to test later whether $\pnf{U_\sigma} < \pnf{U_\rho}$,
it is enough to remember the \llex ordering between prefixes
of the same length, and in case they are equal we can compare 
suffixes. During the proof we will point out where this is used.

\begin{enumerate}[1.)]
\item  For $V=0$ we let $V_0= 0$.
\item Let $V=UW$ and $ U \neq 0 \neq W$.
We have
\begin{align*}
\pnf{UWw} &= \pnf{U_\sigma}  \pnf{\sigma W w}\\
&= \pnf{U_\sigma}  \pnf{ W \sigma w}
\\
&= \pnf{U_\sigma} \pnf{W_\tau} \pnf{(\sigma + \tau) w}
\end{align*}
for some $\sigma\in R(U)$ and $\tau\in R(W)$. 
Thus, for each $\rho$ we have:
$$V_\rho = \min\set{U_\sigma W_\tau}{\rho = \sigma + \tau, \sigma \in R(U), \tau \in R(W)}.$$ 

The minimum search is a little tricky, because $U_\sigma$'s
may have different length so that we are forced to scan through
the word $W_\tau$. The overall time we need for these comparisons 
can be bounded however by $\Oh(s(U){\Abs W})$ by similar techniques 
as used above in other proofs. (For polynomial time results
such a tuning is not necessary, and can be omitted.) 
For creating  smaller tables we used inductively the time 
$\Oh(s(U)({\Abs U}+s(U)) + s(W)({\Abs W}+s(W)).$
We need another term $\Oh(s(U)s(W))$
for computing the length $\abs{\pnf{V_\rho}}$ and, in any case, 
we consider all $\sigma \in R(U)$ and  all $\tau \in R(V)$.

The time we need to do this for all $\rho \in R(V)$ is bounded
by $$\Oh(s(U){\Abs W} + s(U)({\Abs U}+s(U)) + s(W)({\Abs W}+s(W)) + s(U)s(W)).$$
This is within our time bound, because 
\begin{align*}
s(U){\Abs W} + s(U)&({\Abs U}+s(U)) + s(W)({\Abs W}+s(W)) + s(U)s(W)\\ 
&\leq (s(U) +s(W))({\Abs U} + \Abs{W}+ s(U) +s(W))\\
&= s(V)({\Abs V}+s(V)).
\end{align*}

\item Let $V = \alp T U \bet t$. 
We have 
\begin{align*}
\pnf{\alp T U \bet tw } 
&= \alp' T \pnf{U(\eps q + \bet) tw}\\
&= \alp' T U_\sigma\pnf{(\sigma  + \eps q + \bet) tw}
\end{align*}
for some $\sigma \in R(U)$ and some
$\alp'$ with $\alp = \eps p + \alp'$ and $\abs{\alp'} \leq p$ and $\abs \eps \leq 1$. 

Now, 
$$\pnf{(\sigma + \eps q + \bet) tw)} = \bet' t 
\pnf{\rho w}$$
for some  $\bet'$ and  $\rho$ 
such that
$\sigma + \eps q+ \bet = \mu q + \bet'$ and 
$\abs{ \bet'} < q$ and $\rho = \mu p$. 

Thus, for each $\rho$ we may define:
$$V_\rho = \min \oneset{\alp' T U_\sigma \bet' t}$$
where the minimum is taken over all 
$\sigma \in R(U)$,
$\abs{\alp'} \leq p$, and  $\abs{ \bet'} < q$ 
satisfying: 
$$
\alp = \eps p + \alp' \text{ and }
\sigma + \eps q+ \bet = \mu q + \bet' \text{ and }
 \rho = \mu p.
$$
The time we need to do  this for all $\rho \in R(V)$ is bounded
by $$\Oh(1 + s(U)({\Abs U}+s(U)) + s(V)).$$
This is within our time bound, because $s(U)= s(V)$ 
and ${\Abs V} \geq 2 + {\Abs U}$. 
\end{enumerate}
\end{proof}

\begin{corollary}\label{gunnar}Let $ p \mid q$.
We can design a quadratic time algorithm running in time 
$\Oh(s(V)({\Abs V}+s(V))$ which solves the following problem.

{\textbf {Input:} }A valley $v$

{\textbf {Problem:} }Compute the 
\peaknf $\pnf{v}$.
\end{corollary}

\begin{proof}
The \Breduction of a valley is a valley and can be computed in linear
time, keeping the integers in binary notation. 
Again in linear time we find a \sv $V$ and $\gamma \in \Z$ 
such that $v\gq V \gamma$. 
Theorem~\ref{whou} yields 
$$ \pnf{v} = \min \set { V_\rho \llnf{\rho + \gamma}}{\rho \in R(V)}.$$
By Proposition~\ref{fritz} $\llnf{\rho + \gamma}$ can be computed in linear time. 
 
The final step takes time $\Oh(\Abs{V} s(V))$ because $\abs{R(V)} \in \Oh(s(V))$.
\end{proof}

\subsection{Proof of Theorem~\ref{thm:main}}\label{thisistheendmyfriend}
We are now in a position to complete the proof of the main Theorem~\ref{thm:main}. Let $w\in\oneset{a,A,t,T}^\ast$ be the input word. First we rewrite $w$ as a word in $\FP$ and compute its \Breduction $u$ in linear time. We have $\Abs{u} \leq \Abs{w}$ and $s(u) \leq s(w)$. 
Next, $u$ is partitioned into $\alp_1t\cdots\alp_ktDT\bet_1\cdots T\bet_m$, where $D$ is horocyclic or \di. Again, we have 
$\Abs{D} \leq \Abs{u}$ and $s(D) \leq s(u)$.
In any case, the peak of $u$ is inside $D$, so, using the reduction given in Theorem~\ref{thm:otto}, 
 it is sufficient to compute the \peaknf of 
a constant number of horocyclic or \di words $D'$ 
with $\Abs{D'} \leq \Abs{D}+r$ and $s(D') = s(D)$ and $r$ is a constant.
{}From that we get the \peaknf of $u$ in linear time. 

If $D'$ is horocyclic, the \peaknf equals the \llex normal form which can be found in linear time, as demonstrated in Proposition~\ref{fritz}. 

For the case where $D'$ is \di, we have shown in Section~\ref{sec:rtov}, that we can reduce the computation of the \peaknf 
to the computation of two valleys. The reduction takes linear time. Finally {\peaknf}s for all valleys under consideration  can be found in quadratic time 
$\Oh(s(w)({\Abs w}+s(w))$
using Corollary~\ref{gunnar}. \hfill$\square$

\section*{Conclusion}\label{sec:c}
We have seen how to compute the geodesic \peaknf 
in quadratic time in the case $ p \mid q$. Actually, we have shown that 
the uniform problem is decidable in polynomial time. The uniform 
problem takes as input a word $ w \in \oneset{t,t^{-1},a,a^{-1}}^*$ and two
integers $p$ and $q$ written in unary with $ p \mid q$.
As the  values $r$ defined in Eqs.~\ref{eq:fred} and~\ref{const}
are polynomial in $q$, our quadratic time non-unifom algorithm yields
a polynomial time uniform algorithm. 

For the general case, where $p$ does not divide $q$, it remains
to compute geodesics for \di words, i.e., those, for which the
$t$-sequence of the \Breduction starts with $t^{-1}$ and ends with  $t$.

It is open whether the remaining problem is in P or whether it is
Co-NP complete or whether the complexity is somewhere in between. 

\addcontentsline{toc}{section}{References}\nopagebreak

\begin{thebibliography}{1}

\bibitem{baumslag62some}
G.~Baumslag and D.~Solitar.
\newblock {Some two-generator one-relator non-Hopfian groups}.
\newblock {\em Bull. Amer. Math. Soc.}, 68:199--201, 1962.

\bibitem{ddm09}
V.~Diekert, A.~J. Duncan, and A.~Miasnikov.
\newblock Geodesic rewriting systems and pregroups.
\newblock In O.~Bogopolski, I.~Bumagin, O.~Kharlampovich, and E.~Ventura,
  editors, {\em Combinatorial and Geometric Group Theory}, Trends in
  Mathematics. Birkh\"auser, 2009.
\newblock To Appear.

\bibitem{elder09linear}
M.~Elder.
\newblock {A linear-time algorithm to compute geodesics in solvable
  {B}aumslag-{S}olitar groups}, 2009.
\newblock Preprint, arXiv.org:0903.0216.

\bibitem{elder09some}
M.~Elder and A.~Rechnitzer.
\newblock {Some geodesic problems in groups}, 2009.
\newblock Preprint, arXiv.org:0907.3258.

\bibitem{fredenadams09}
E.~Freden and J.~Adams.
\newblock {A context-sensitive combing associated with {B}aumslag-{S}olitar
  2,7}, 2009.
\newblock Preprint.

\bibitem{fcs09}
E.~Freden, T.~Cawley, and J.~Schofield.
\newblock {Growth in {B}aumslag-{S}olitar groups {I}: {S}ubgroups and
  rationality}, 2009.
\newblock Preprint,.

\bibitem{jan88eatcs}
M.~Jantzen.
\newblock {\em Confluent String Rewriting}, volume~14 of {\em EATCS Monographs
  on Theoretical Computer Science}.
\newblock Springer-Verlag, 1988.

\bibitem{MR42:4343}
W.~Kuich.
\newblock On the entropy of context-free languages.
\newblock {\em Information and Control}, 16:173--200, 1970.

\bibitem{ls77}
R.~E. Lyndon and P.~E. Schupp.
\newblock {\em Combinatorial group theory}.
\newblock Springer-Verlag, Heidelberg, 1977.

\end{thebibliography}

\newcommand{\Ju}{Ju}\newcommand{\Ph}{Ph}\newcommand{\Th}{Th}\newcommand{\Yu}{Y%
u}\newcommand{\Zh}{Zh}

\newpage
\section*{Appendix}

\subsection*{Example: Computing \llexnf{}s of horocyclic elements}\label{ex:horo}
We give an example to illustrate how our algorithm for finding the \llexnf of horocyclic elements works. Suppose we have $p=1$, $q=3$ and we want to find the \llexnf of the word (of length 63) 
\begin{align*}
w & =a^{7}ta^{14}t^{-1}a^{-2}t^{2}a^{9}t^{-1}a^{2}t^{-1}a^{23} \\
& = a^{7}ta^{14}TA^{2}t^{2}a^{9}Ta^{2}Ta^{23}\\ & =
7t14T{-2}tt9T2T23.
\end{align*}

We proceed in several steps. 

\subsubsection*{Precomputation and finding a suitable slope}
Three \Breduction steps yield the integer $w\gq 157$. The greedy algorithm of Proposition \ref{fritzchen} gives $w\gq t52T1\gq t^217T1T1\gq t^35T2T1T1$, hence $\llnf{w}=t^3\llnf{u}$, for the slope $u=5T2T1T1$. 

\subsubsection*{Computing the \llexnf for a slope}
We now regard the slope $u=5T2T1T1$. Since $p=1$ and $q=3$, we have $r=4$. We use the algorithm described in the proof of Proposition~\ref{fritz}. Table~\ref{tab:horo1} shows the intermediate results of the unoptimized algorithm as presented in the first half of the proof. Note that the algorithm never has to store more than two columns of the table at the same time. So, due to the length of the entries, linear space is needed. 

\begin{table}[ht]
$$\begin{array}{r|l|l|l|l|}
\rho&\llnf{\rho}&\llnf{u(1,\rho)}&\llnf{u(2,\rho)}&\llnf{u(3,\rho)}\cr\hline
-9&tt{-1}TT&\cellcolor{BSgray}&\cellcolor{BSgray}&\cellcolor{BSgray}\cr\hline
\vdots&&\cellcolor{BSgray}&\cellcolor{BSgray}&\cellcolor{BSgray}\cr\hline
-4&t{-1}T{-1}&t1TT2&t1T2TT2&t2TT{-1}TT{-1}\cr\hline
-3&t{-1}T&t1T1T&t1T2T1T&t2TT{-1}TT\cr\hline
-2&{-2}&t1T1T1&t1T2T1T1&t2TT{-1}TT1\cr\hline
-1&{-1}&t1T1T2&t2TT{-1}T{-1}&t2TTT{-2}T{-1}\cr\hline
0&0&t1T2T&t2TT{-1}T&t2TTT{-2}T\cr\hline
1&1&t1T2T1&t2TTT{-2}&t2TTT{-1}T{-2}\cr\hline
2&2&t2TT{-1}&t2TTT{-1}&t2TTT{-1}T{-1}\cr\hline
3&t1T&t2TT&t2TTT&t2TTT{-1}T\cr\hline
4&t1T1&t2TT1&t2TTT1&t2TTTT{-2}\cr\hline
5&t1T2&\cellcolor{BSgray}&\cellcolor{BSgray}&\cellcolor{BSgray}\cr\hline
6&t2T&\cellcolor{BSgray}&\cellcolor{BSgray}&\cellcolor{BSgray}\cr\hline
7&t2T1&\cellcolor{BSgray}&\cellcolor{BSgray}&\cellcolor{BSgray}\cr\hline
8&tt1TT{-1}&\cellcolor{BSgray}&\cellcolor{BSgray}&\cellcolor{BSgray}\cr\hline
9&tt1TT&\cellcolor{BSgray}&\cellcolor{BSgray}&\cellcolor{BSgray}\cr\hline
\end{array}$$
\label{tab:horo1}
\caption{Intermediate resulte during the computation of $\llnf{5T2T1T1}$}
\end{table}

Now let's consider the optimized variant from the second part fo the proof. The computational steps are essentially the same, but only a constant amount of data is stored. For example for $i=1$ we have $\abs{p_{1,\rho}}=\min\{\llnf{u(1,\tau)}:\tau\in R\}=5$, so after having computed the column $\llnf{u(1,\rho)}$ in Table~\ref{tab:horo1}, we can cut off the first 5 letters of every entry, save their lexicographical order as an order on $R$ ($-4<-3=-2=-1<0=1<2=3=4$), and output them. 
The complete output of the algorithm is shown in Table~\ref{tab:horo2}. 

\begin{table}[ht]
$$\begin{array}{r|l|ll|ll|}
\rho&i=1&i=2&&i=3&\cr\hline
-4&t1TT1&T,&0&TT{-1},&0\cr\hline
-3&t1T1T&T,&1&TT,&0\cr\hline
-2&t1T1T&T,&1&TT1,&0\cr\hline
-1&t1T1T&{-1},&2&{-2}T{-1},&1\cr\hline
0&t1T2&{-1},&2&{-2}T,&1\cr\hline
1&t1T2&T,&3&{-1}T{-2},&2\cr\hline
2&t2TT&T,&3&{-1}T{-1},&2\cr\hline
3&t2TT&T,&3&{-1}T,&2\cr\hline
4&t2TT&T,&3&T{-2},&3\cr\hline
\end{array}$$
\label{tab:horo2}
\caption{Output of the optimized algorithm for $\llnf{5T2T1T1}$}
\end{table}

Following the pointers and reading off the result, starting at $(\rho,i)=(1,3)$, gives $\llnf{u}=t2TTT{-1}T{-2}$. Thus, we have
$$\llnf{w}=t^3\llnf{u}=t^42TTT{-1}T{-2} = t^4a^2t^{-3}a^{-1}t^{-1}a1^{-2}.$$


\end{document}